\documentclass[12pt]{article} 

\usepackage{amsmath}
\usepackage{amsthm}
\usepackage{amsfonts}
\usepackage{mathrsfs}
\usepackage{stmaryrd}
\usepackage{setspace}
\usepackage{fullpage}
\usepackage{amssymb}
\usepackage{breqn}
\usepackage{enumitem}
\usepackage{bbold} 
\usepackage{authblk}
\usepackage{comment}
\usepackage{hyperref}
\usepackage{pgf,tikz}
\usepackage{graphicx}
\usepackage{xcolor}
\usetikzlibrary{decorations.pathreplacing,arrows}
\tikzstyle{vertex}=[circle,draw=black,fill=black,inner sep=0,minimum size=3pt,text=white,font=\footnotesize]

\newtheorem*{rep@theorem}{\rep@title}
\newcommand{\newreptheorem}[2]{%
\newenvironment{rep#1}[1]{%
 \def\rep@title{#2 \ref{##1}}%
 \begin{rep@theorem}}%
 {\end{rep@theorem}}}
\makeatother

\theoremstyle{plain}
\newtheorem{theorem}{Theorem}[section]
\newreptheorem{theorem}{Theorem}
\newtheorem{lemma}[theorem]{Lemma}
\newtheorem{corollary}[theorem]{Corollary}

\newtheorem{proposition}[theorem]{Proposition}

\theoremstyle{definition}

\newtheorem{defn}[theorem]{Definition}

\newtheorem{claim}[theorem]{Claim}

\newtheorem*{proposition*}{Proposition}

\newcommand\ex{\ensuremath{\mathrm{ex}}}
\newcommand\inj{\ensuremath{\mathrm{inj}}}

\newcommand\cC{{\mathcal C}}

\newcommand\cF{{\mathcal F}}
\newcommand\cG{{\mathcal G}}

\newcommand\cN{{\mathcal N}}
\newcommand\cP{{\mathcal P}}

\newcommand{\ignore}[1]{}

\title{Stability from graph symmetrization arguments in generalized Tur\'an problems}

\author{D\'aniel Gerbner\footnote{Alfr\'ed R\'enyi Institute of Mathematics, E-mail: \texttt{gerbner@renyi.hu.}}, Hilal Hama Karim\footnote{Department of Computer Science and Information Theory, Faculty of Electrical Engineering and Informatics, Budapest University of Technology and Economics, Műegyetem rkp. 3., H-1111 Budapest, Hungary. E-mail: \texttt{hilal.hamakarim@edu.bme.hu}}}

\date{}

\begin{document}
\maketitle
\begin{abstract}

    Given graphs $H$ and $F$, $\ex(n,H,F)$ denotes the largest number of copies of $H$ in $F$-free $n$-vertex graphs. Let $\chi(H)<\chi(F)=r+1$. We say that $H$ is \textit{$F$-Tur\'an-stable} if the following holds. For any $\varepsilon>0$ there exists $\delta>0$ such that if an $n$-vertex $F$-free graph $G$ contains at least $\ex(n,H,F)-\delta n^{|V(H)|}$ copies of $H$, then the edit distance of $G$ and the $r$-partite Tur\'an graph is at most $\varepsilon n^2$. 
We say that $H$ is \textit{weakly $F$-Tur\'an-stable} if the same holds with the Tur\'an graph replaced by any complete $r$-partite graph $T$. It is known that such stability implies exact results in several cases.

 We show that complete multipartite graphs with chromatic number at most $r$ are weakly $K_{r+1}$-Tur\'an-stable. Partly answering a question of Morrison, Nir, Norin, Rza{\.z}ewski and Wesolek positively, we show that for every graph $H$, if $r$ is large enough, then $H$ is $K_{r+1}$-Tur\'an-stable. Finally, we prove that if $H$ is bipartite, then it is weakly $C_{2k+1}$-Tur\'an-stable for $k$ large enough.
\end{abstract}


\vskip3mm
\textbf{Corrections to the previous version}: We would like to thank Casey Tompkins and D\"om\"ot\"or P\'alv\"olgyi for pointing out some mistakes in the previous version that was published in Journal of Graph Theory (\textbf{107}(4),pp:681-892, 2024). Those mistakes are fixed in this current version. Particularly, the proof of Theorem \ref{multipa} and the definition of $F^\star$ in Section \ref{sec4} have been revised. 

\section{Introduction}

One of the central problems in extremal graph theory is giving bounds on
$\ex(n,F)$, which is the largest number of edges in $n$-vertex $F$-free graphs. Tur\'an's theorem \cite{turan} states that $\ex(n,K_{r+1})=|E(T(n,r))|$, where the \textit{Tur\'an graph} $T(n,r)$ is the complete $r$-partite graph with each part of order $\lfloor n/r\rfloor$ or $\lceil n/r\rceil$.

A natural generalization is to count other subgraphs instead of edges. Given graphs $H$ and $G$, let
$\cN(H,G)$ denote the number of isomorphic copies of $H$ in $G$. Given graphs $H$ and $F$, let $\ex(n,H,F)$ denote the largest $\cN(H,G)$ among the $n$-vertex $F$-free graphs. After several sporadic results, the systematic study of this quantity was initiated by Alon and Shikhelman \cite{alon}.

These problems are usually called \textit{generalized Tur\'an problems}.
The first such result is due to Zykov \cite{zykov}, who showed that $\ex(n,K_k,K_{r+1})=\cN(K_k,T(n,r))$. Let us briefly describe his proof. Given two vertices $u$ and $v$ of a graph $G$, we say that we \textit{symmetrize} $u$ to $v$ if we delete all the edges incident to $u$ and add all the edges of form $uw$ where $w$ is a neighbor of $v$. In other words, we change the neighborhood of $u$ to the neighborhood of $v$. 

If $u$ and $v$ are non-adjacent and $G$ is $K_{r+1}$-free, then the resulting graph $G'$ is also $K_{r+1}$-free. Indeed, a copy of $K_{r+1}$ would contain $u$ as all the new edges are incident to $u$. Then this copy cannot contain $v$, as $v$ is not adjacent to $u$. But then deleting $u$ and adding $v$ to this copy, we find another copy of $K_{r+1}$ that is also present in $G$, contradicting our assumption.

The other key property is that if $u$ is contained in $x$ copies of $K_k$ and $v$ is contained in $y$ copies of $K_k$, then this symmetrization removes $x$ and adds $y$ copies of $K_k$. Therefore, by applying the symmetrization if $x\le y$, the total number of copies of $K_k$ does not decrease.

The proof goes on by applying several such symmetrization steps. One can show that this process terminates, i.e., at one point we arrive to a graph where symmetrization cannot change the graph. This means that non-adjacent vertices have the same neighborhood, thus being non-adjacent is an equivalence relation, i.e., the resulting graph is complete multipartite (obviously with at most $r$ parts). We omit the proof that the process terminates, as it is not relevant to our work. We also omit for the same reason showing that among complete multipartite graphs with at most $r$ parts, $T(n,r)$ contains the most copies of $K_k$.

Gy\H ori, Pach and Simonovits \cite{gyps} generalized this argument, showing that for any complete multipartite graph $H$, $\ex(n,H,K_{r+1})=\cN(H,T)$ for some complete $r$-partite graph $T$ (which is not necessarily the Tur\'an graph). One could think that this is the limit of Zykov's symmetrization argument in this topic, since only cliques have the property that symmetrization cannot create them, and only complete multipartite graphs $H$ have the property that symmetrizing either $u$ to $v$ or $v$ to $u$ does not decrease the number of copies of $H$. However, some more advanced applications have appeared in the literature. One can symmetrize only to vertices $v$ satisfying some property that ensures that no $F$ will be created \cite{LW,gp}. Another example is \cite{gerb2}, where the neighborhood of $u$ is changed to the common neighborhood of a set of vertices.

Here we will present three stability results on generalized Tur\'an problems that are obtained by using symmetrization. Stability refers to the phenomenon that an $F$-free $n$-vertex graph with almost $\ex(n,H,F)$ copies of $H$ is close to the extremal graph. There are different kinds of stability, depending on what ``almost'' and ``close'' mean in the previous sentence. Here we deal with one specific kind of stability.

The \textit{edit distance} of two $n$-vertex graphs $G$ and $G'$ is the number of edges needed to be deleted and added to $G$ in order to obtain a graph isomorphic to $G'$. Let $h=|V(H)|$.
Given a graph $F$ with chromatic number $r+1$ and another graph $H$ with chromatic number at most $r$, we say that $H$ is \textit{$F$-Tur\'an-stable} if the following holds. For any $\varepsilon>0$ there exists $\delta>0$ such that if an $n$-vertex $F$-free graph $G$ contains at least $\ex(n,H,F)-\delta n^h$ copies of $H$, then the edit distance of $G$ and $T(n,r)$ is at most $\varepsilon n^2$. 
We say that $H$ is \textit{weakly $F$-Tur\'an-stable} if the same holds with $T(n,r)$ replaced by any complete $r$-partite graph $T$.

Using these notions, the classical Erd\H os-Simonovits stability theorem \cite{erd1,erd2,simi} shows that $K_2$ is $K_{r+1}$-Tur\'an-stable for every $r\ge 2$. The first stability result concerning generalized Tur\'an problems is due to Ma and Qiu \cite{mq}, who showed that $K_k$ is $K_{r+1}$-Tur\'an-stable for every $r\ge k$.
Tur\'an-stable graphs were introduced by the first author in \cite{gerb2}. It was shown in \cite{gerb2} that if $H$ is (weakly) $F$-Tur\'an-stable, it often implies that $H$ is (weakly) $F'$-Tur\'an-stable for some other graphs $F'$. 

Another important feature of these notions is that they often imply exact results on $\ex(n,H,F)$.
We say that $H$ is \textit{$F$-Tur\'an-good} if for each sufficiently large $n$ we have $\ex(n,H,F)=\cN(H,T(n,r))$. We say that $H$ is \textit{weakly $F$-Tur\'an-good} if for each sufficiently large $n$ we have that $\ex(n,H,F)=\cN(H,T)$ for some complete $r$-partite graph $T$. It was shown in \cite{gerb2} that if $H$ is weakly $F$-Tur\'an-stable and $F$ has a color-critical edge (an edge whose deletion decreases the chromatic number), then $\ex(n,H,F)=\cN(H,T)$ for some complete $r$-partite graph $T$. Some other exact results were shown in \cite{gerbi}.

The first author in \cite{gerb3} claimed that a result of Liu, Pikhurko, Sharifzadeh and Staden \cite{perfstb} implies that complete multipartite graphs are weakly $K_{r+1}$-Tur\'an-stable. This is not true. Even though our problem is in many sense simpler than theirs, it still does not fit into their general framework. Here we present a proof of this statement. Note that the case $H=C_4$ was proved in \cite{hh}.

\begin{theorem}\label{multipa}
    Let $H$ be a complete $k$-partite graph and $r\ge k$. Then $H$ is weakly $K_{r+1}$-Tur\'an-stable.
\end{theorem}

Morrison, Nir, Norin, Rza{\.z}ewski and Wesolek \cite{mnnrw} showed (proving a conjecture of Gerbner and Palmer \cite{gerpal}) that for every graph $H$, if $r$ is large enough, then $H$ is $K_{r+1}$-Tur\'an-good. They ask whether this can be improved to $K_{r+1}$-Tur\'an-stability. We partly answer this question in the affirmative.

\begin{theorem}\label{event}
    For every graph $H$, if $r\ge 300h^9$, then $H$ is $K_{r+1}$-Tur\'an-stable.
\end{theorem}


Finally, we obtain a theorem of a similar flavor. 

\begin{theorem}\label{cycles}
    For every bipartite graph $H$, if $k>(2h)^{h+1}$, then
    $H$ is weakly $C_{2k+1}$-Tur\'an-stable. 
\end{theorem}

We note that the above theorem implies that $H$ is weakly $C_{2k+1}$-Tur\'an-good, since $C_{2k+1}$ has a color-critical edge. The first author \cite{gerb3} showed for every $k$ a graph that is $C_{2k+1}$-Tur\'an-good and not weakly $C_{2\ell+1}$-Tur\'an-good for any $\ell<k$.


Let us briefly summarize what is implied by our theorems. First, the $F$-Tur\'an-stability implies $F'$-Tur\'an-stability for other graphs, as observed in \cite{gerb2} in the case $F=K_{r+1}$. We present the following more general version. 
We will be using the graph removal lemma \cite{rsz} in the proof of the next proposition as well as that of Theorem \ref{cycles}. Recall that it states that if a graph $G$ contains $o(n^{|V(H)|})$ copies of a graph $H$, then by removing $o(n^2)$ edges we can remove all the copies of $H$ from $G$.


\begin{proposition}
    Let $\chi(F)=\chi(F')=r+1$ and assume that $F'$ is contained in a blowup of $F$. If $H$ is weakly $F$-Tur\'an-stable, then $H$ is weakly $F'$-Tur\'an-stable.
\end{proposition}

A blowup of a graph is obtained by repeatedly replacing a vertex $v$ with more vertices $v_1,\dots,v_k$ and each edge $uv$ by edges $uv_1,\dots,uv_k$.

\begin{proof}
    Let $G$ be an $F'$-free $n$-vertex graph with at least $\ex(n,H,F')-o(n^h)$ copies of $H$. Since $H$ is weakly $F$-Tur\'an-stable, for any $F$-free graph $G_0$ with $\ex(n,H,F)$ copies of $H$, there is a complete $r$-partite graph $T$ with edit distance $o(n^2)$ to $G_0$. This implies that $\cN(H,T)=\ex(n,H,F)-o(n^h)$. Note that $\ex(n,H,F')\ge \cN(H,T)$ because $T$ is $F'$-free.
    
    By a result of Alon and Sikhelman \cite{alon} we have $\ex(n,F,F')=o(n^{|V(F)|})$. By the removal lemma, there are $o(n^2)$ edges in $G$ contained in each copy of $F$. We delete those edges to obtain an $F$-free graph $G'$. Those edges are in $o(n^h)$ copies of $H$, thus $\cN(H,G')=\cN(H,G)-o(n^h)\ge \ex(n,H,F')-o(n^h)-o(n^h)\ge \cN(H,T)-o(n^h)\ge \ex(n,H,F)-o(n^h)$ using the already established bounds. Since $H$ is weakly $F$-Tur\'an-stable, there is a complete $r$-partite graph $T'$ with edit distance $o(n^2)$ to $G'$. Then the edit distance of $T'$ and $G$ is also $o(n^2)$, completing the proof.
\end{proof}

There are multiple ways in \cite{gerb3} and \cite{gerbi} to obtain new (weakly) $F$-Tur\'an-stable graphs from other $F$-Tur\'an-stable graphs. Our results give new building blocks to those methods. Finally, as we have mentioned, our new stability results give new exact results using theorems from \cite{gerb3} and \cite{gerbi}.

The rest of the paper is organized as follows. Each theorem is proved in its own section, and we finish the paper with some concluding remarks.

\section{Proof of Theorem \ref{multipa}} \label{sec2}
Let us denote the edit distance between two graphs $G_1$ and $G_2$ on the same vertex set by $\Delta_1(G_1,G_2):=|E(G_1) \triangle E(G_2)|$. Let $H$ be a complete $k$-partite graph on $h$ vertices and $k\le r$. We first prove a lemma that will be used in the proof.


\begin{lemma}\label{partssizes}
For every $r\geq k$ there is $\gamma>0$ and $\zeta >0$ such that the following holds. For sufficiently large $n$, if $P$ is a complete $s$-partite $K_{r+1}$-free graph on $n$ vertices such that $\cN(H,P)\geq\ex(n,H,K_{r+1})-\zeta \cdot n^h$, then $s=r$ and the size of every part of $P$ is at least $\gamma n$.

\end{lemma}


\begin{proof}
Let $1/s \gg \gamma \gg \zeta>0$. Assume that there is some part $V_j$ with $|V_j|<\gamma n$. Let $V_1$ be the largest part, $|V_1|=pn$, thus $p\ge 1/s$.  Let $H[A\cup B]$ be an induced bipartite subgraph of $H$ with $|A|=a$ and $|B|=b$ such that $A \subseteq V_1$ and $B \subseteq V_j$. Let us move $\lfloor pn/2\rfloor$ vertices of $V_1$ to $V_j$. Then the number of copies of such bipartite subgraphs between $V_1$ and $V_j$ increases by at least 

$$\binom{\lfloor pn/2\rfloor}{a}\binom{\lceil pn/2\rceil}{b}-\left[\binom{\gamma n}{a}\binom{pn}{b}+\binom{pn}{a}\binom{\gamma n}{b}\right]\geq c_1n^{a+b}-\gamma c_2n^{a+b}\geq \frac{c_1}{2}n^{a+b},$$

where $c_1$ and $c_2$ are constants. The last inequality holds because $c_1$ depends on $p, a$ and $b$ but not $\gamma$, and $\gamma \ll p$. 
Note that if a copy of $H$ does not intersect both $V_1$ and $V_j$, then their number will not be affected. 
Consequently, the number of copies of $H$ increases by at least 
$c n^h$ (where $c$ is a constant), a contradiction.

This argument proves $s=r$, too. Since we know $s\leq r$ as $P$ is $K_{r+1}$-free, and if $s<r$ we may think of $r-s$ parts being empty, so contain $< \gamma n$ vertices.
\end{proof}

Now, let us begin the proof of Theorem \ref{multipa}. Let $\gamma\gg \alpha\gg \beta\gg \mu \gg \varepsilon\gg\delta$ be a sequence of positive numbers such that each depends on $H$, $r$ and the previous numbers in the sequence, and is sufficiently small.
We will also assume that $n$ is sufficiently large.





Assume the theorem does not hold, and hence  there is a graph $G$ with $|V(G)|=n$, for a sufficiently large $n$, such that $\cN(H,G) \geq \ex(n,H,K_{r+1})- \delta n^h$ while $\Delta_1(G,F)>\varepsilon n^2$, for every complete $r$-partite graph $F$. 
Now we apply Zykov symmetrization repeatedly. It was shown in \cite{gyps} that we can pick symmetrization steps in such a way that the number of copies of $H$ does not decrease, no $K_{r+1}$ is created and the process eventually terminates, which results in a complete multipartite graph. This gives a sequence of graphs $G=G_0,G_1, \ldots, G_l$, where $\cN(H,G_{i+1}) \geq \cN(H,G_i)$ for all $i=0,1,\ldots, l-1$, and $G_l$ is a complete $s$-partite graph. 
In particular, $\cN(H,G_l) \geq \cN(H,G)\geq \ex(n,H,K_{r+1})-\delta n^h$, and hence, by Lemma \ref{partssizes}, $s=r$, and every part of $G_l$ has size at least $\gamma n$.

Let $t$ be the largest $i$ such that $\Delta_1(G_i, F) > \varepsilon n^2$ for every complete $r$-partite graph $F$. Let $T$ be the closest graph to $G_t$ and $T'$ be the closest graph to $G_{t+1}$ among all complete $r$-partite graphs on the vertex set $V(G)$.  Observe that $\varepsilon n^2<\Delta_1(G_t, T)\le \Delta_1(G_t, T')\le\Delta_1(G_t, G_{t+1})+\Delta_1(G_{t+1}, T') \le 2n+\varepsilon n^2 \le 2\varepsilon n^2$. Let $G':=G_t$.




Note that this implies $e(G')-e(T) \leq 2\varepsilon n^2$. Since each edge is in at most $n^{h-2}$ copies of $H$, we have $T$ has at least $\cN(H,G')-2\varepsilon n^h$ copies of $H$. Thus, $\cN(H,T)\geq \ex(n,H,K_{r+1})- (\delta + 2 \varepsilon) n^h$. By Lemma \ref{partssizes} again, we have that each part of $T$ has size  at least $\gamma n$.

Let $V_1, V_2, \ldots, V_r$ be the parts of $T$. 


\begin{claim}\label{clammnew}
In the graph $G'$, every vertex $v$ has at least $\beta n$ non-neighbors in its part.
\end{claim}
\begin{proof}[Proof of the claim] 

Assume indirectly that $v\in V_j$ has fewer than $\beta n$ non-neighbors in $V_j$ (thus $v$ has at least $\beta n$ neighbors in $V_j$). Assume first that $v$ has at least $\beta n$ neighbors in each other part and for every $i=1,2,\ldots, r$, choose $\beta n$ neighbors of $v$ from $V_i$, to form an induced subgraph $G''$ of $G'$. Clearly, $G''$ does not contain a $K_r$, as together with $v$, this would create a $K_{r+1}$ in $G'$. Now, we have that $G''$ is a $K_r$-free graph and $|V(G'')|:=n'=\beta rn$. Then, by Tur\'an's theorem we have 
$$e(G'') \leq \left(\frac{r-2}{r-1}\right)\frac{(n')^2}{2}=\left(\frac{r-1}{r} - \frac{1}{r(r-1)}\right)\frac{(n')^2}{2}=\frac{r-1}{r}\frac{(n')^2}{2}-cn^2,$$
where $c=\frac{r}{2(r-1)} \beta^2$. On the other hand, as $G''$ is an induced subgraph of $G'$, the vertex set of $G''$ consists of $r$ sets, each of them is a subset of size $\beta n$ of a different part of $T$, and $\Delta_1(G',T) \leq 2\varepsilon n^2$, we have
$$e(G'') \geq e(T[V(G'')])-2 \varepsilon n^2=\frac{r-1}{r}\frac{(n')^2}{2}-2 \varepsilon n^2.$$
Thus, we have $cn^2 \leq 2\varepsilon n^2$, which contradicts the assumption that $\varepsilon$ is sufficiently small compared to $\beta$.

Assume now that $v$ has fewer than $\beta n$ neighbors in $V_i$. Then we move $v$ to $V_i$. We claim that the complete $r$-partite graph $T^*$ we obtain this way is closer to $G'$ than $T$, a contradiction. Indeed, $v$ was incident to more than $|V_j|-\beta n$ edges inside its part $V_j$ and more than $|V_i|-\beta n$ missing edges with other endpoint inside $V_i$ in $T$. In $T^*$, $v$ is incident to fewer than $\beta n$ edges inside its part, now $V_i$, and fewer than $\beta n$ missing edges with other endpoints inside $V_j$.
\end{proof}


Now there are two cases to consider, and we shall see both of them lead to contradictions.

\vskip3mm

\textbf{Case 1.} Every vertex has fewer than $\alpha n$ non-neighbors in each of the parts other than its own part. 

First, this implies there are no edges inside any of the parts, for if there is an edge inside some part, then we will have a $K_{r+1}$ in $G'$. To see this, let $v_0v_1$ be an edge inside some part, we may say $V_1$. Then, $v_0$ and $v_1$ have at least $\gamma n-2\alpha n$ common neighbors in $V_2$, choose $v_2$ among them, and then continuing this way, we can choose the vertex $v_{r-1}$ from $V_{r-1}$ forming a $K_{r}$ in the first $r-1$ parts.  The so far chosen $r$ vertices have at least $\gamma n - r\alpha n$ common neighbors in $V_r$, from which we can choose a vertex $v_r$ to form a $K_{r+1}$.

Therefore, $G'$ differs from $T$ only by missing some edges between the different parts. First we show that each edge $uv$ with $u$ and $v$ from different parts is contained in at least $(\gamma n/2)^{h-2}$ copies of $H$ in $T$. 



Indeed, for a copy of $H$ that contains $u$ and $v$, we need to choose the other $h-2$ vertices of $H$. As we choose each one from some part of $T$ and each part of $T$ contains at least $\gamma n$ vertices, there are at least $\gamma n-h \geq \gamma n/2$ choices for each of the  other $h-2$ vertices. Here, $-h$ generously excludes the previously chosen vertices.

Let us now bound $\cN(H,T)-\cN(H,G')$. The more than $\varepsilon n^2$ missing edges give at least $\varepsilon(\gamma/2)^{h-2}n^h$ missing copies of $H$, but some $H$ may be counted multiple times. Those copies contain at least 2 missing edges. There are two cases. For two missing independent edges, we can present the upper bound $4\varepsilon^2n^h$ by picking two edges and four vertices in at most $4\varepsilon^2n^4$ ways, and then each other vertex in at most $n$ ways. For two missing edges sharing a vertex, we can present the upper bound $4\varepsilon \alpha n^h$ by picking a missing edge, an endpoint of that edge, another missing edge from the same vertex, and then each other vertex in at most $n$ ways. We obtained that $\cN(H,G')<\cN(H,T)-\varepsilon((\gamma/2)^{h-2}-4\varepsilon-4\alpha) n^h<\cN(H,T)-\delta n^h\le\ex(n,H,K_{r+1})-\delta n^h\le\cN(H,G)$. This is a contradiction, since we obtained $G'$ from $G$ through the symmetrization steps.

\vskip3mm

\textbf{Case 2.} There are some vertices that have at least $\alpha n$ non-neighbors in some of the other parts.

\textbf{Case 2.1.} Each vertex has at most $\beta n$ neighbors in its own part.
Let $B$ be the set of such vertices with at least $\alpha n$ non-neighbors in some other part. As $\Delta_1(G',T) \leq 2\varepsilon n^2$, we have $|B|\alpha n \leq 2\varepsilon n^2$, which means $|B| \leq \frac{2 \varepsilon}{\alpha} n$. Hence, $|V_i \setminus B|=\gamma n - |B_i|\geq \gamma n-\frac{2 \varepsilon}{\alpha} n=\gamma' n$, where $\gamma'=\gamma - \frac{2 \varepsilon}{\alpha}$. 
Note that, for every $i$, each vertex $v \in V_i\setminus B$ has fewer than $\alpha n$ non-neighbors in any of the other parts. As in the previous case, this implies that there are no edges inside $V_i \setminus B$, for every $i$.






For each $v \in B:=\bigcup_{i=1}^r B_i$, delete the at most $\beta n$ edges incident to it inside its part and add the at least $\alpha n$ edges between $v$ and the other parts that are missing in $G'$. 
Note that after all these changes for every vertex in $B$ are done, we obtain an $r$-partite graph $G^*$. In $G^*$, a vertex from $B$ has no non-neighbors in the other parts and a vertex outside $B$ has fewer than $\alpha n$ non-neighbors in any of the other parts.  Let us now observe the change in the number of copies of $H$. For each $v \in B$, we can choose one of its new neighbors in at least $\alpha n$ ways, and then we choose $h-2$ vertices from other parts. In each part $V_i$, we have at least $\gamma n$ vertices, and the $h-1$ vertices picked from the other parts each have at most $\alpha n$ non-neighbors in $V_i$, thus we can pick each vertex at least $\gamma n-(h-1)\alpha n\ge \gamma n/2$ ways. Therefore, the number of new copies of $H$ is at least $\gamma^{h-2}\alpha n^{h-1}/2^{h-1}$, and the number of copies deleted is at most $\beta n^{h-1}$, since we have to pick $v$, one of the (at most) $\beta n$ neighbors of it in its part, and other vertices. 

We repeat this for every vertex $u$ in $B$. Note that again we add at least $\alpha n$ new edges, but it is possible that some of the at most $\beta n$ edges incident to $u$ in its part were already deleted. This does not affect our calculations; the total number of deleted copies of $H$ is at most $\beta |B| n^{h-1}$. Note that a new copy of $H$ can be counted multiple times, but at most $|E(H)|\le h^2$ times. Therefore, the total number of new copies of $H$ is at least $\gamma^{h-2}\alpha |B| n^{h-1}/2^{h-1}h^2$.
So the increase is at least $c_0\alpha n^{h-1}$ for some constant $c_0$, which means $\cN(H,G^*) \geq \cN(H,G')+ |B| c_0\alpha n^{h-1}$. 

Note that if $\Delta_1(G^*,T) \geq \varepsilon n^2/2$, then we can use the same argument as in Case 1. The only difference is that $G'$ was missing at least $\varepsilon n^2$ edges, while $G^*$ is missing at least $\varepsilon n^2/2$ edges. The same calculation gives $\cN(H,G^*)<\cN(H,T)-\varepsilon((\gamma/2)^{h-2}-\varepsilon-2\alpha) n^h<\cN(H,T)-\delta n^h$, a contradiction.
Thus, $\Delta_1(G^*,T) \leq \varepsilon n^2/2$. Hence
$$\varepsilon n^2 < \Delta_1(G',T)\leq \Delta_1(G',G^*)+\Delta_1(G^*,T)\leq \Delta_1(G',G^*)+\varepsilon n^2/2,$$
which implies that the number of edges we added and removed from $G'$ is $\Delta_1(G',G^*) > \varepsilon n^2/2$. Observe that the number of edges added and removed is at most $2\alpha |B|n$. 
Therefore, the increase of the number of copies of $H$ is at least $|B| c_0\alpha n^{h-1}\ge c_0 \varepsilon n^h/4$ if $|B|>0$, which is a contradiction. This completes the proof.

\textbf{Case 2.2.} There is a vertex with more than $\beta n$ neighbors in its part, let $B'$ denote the set of such vertices. As $\Delta_1(G',T) \leq 2\varepsilon n^2$, we have $|B'|\beta n \leq 2\varepsilon n^2$, which means $|B'| \leq \frac{2 \varepsilon}{\beta} n$. Similar to the proof of Claim \ref{clammnew}, we can show that $v$ has fewer than $\mu n$ neighbors in $V_j$ for some $j$. 

Recall that $B$ denotes the set of vertices not in $B'$ with at least $\alpha n$ non-neighbors in some of the other parts. For them, we do the same change as in Case 2.1. 
For each vertex $v\in B'$, we pick an $i=i_v$ such that $v$ has at most $\mu n$ neighbors in $V_i$. Then we move $v$ to $V_i$, delete the at most $\mu n$ edges from $v$ to $V_i$, and add all the edges from $v$ to the part that used to contain $v$. Then we repeat this for the other vertices in $B'$.

The resulting graph $G^{**}$ is $r$-partite. Let us now observe the change in the number of copies of $H$. First, as in Case 2.1, this increase is at least $c_0 \varepsilon n^h/2$ if $|B|>0$. 
For a vertex $u\in B'$, we added new edges connecting it to its non-neighbors in its part. By Claim \ref{clammnew}, there are at least $\beta n$ such vertices. However, it is possible that some of these $\beta n$ non-neighbors of $u$ were moved out during this procedure. This happens at most $|B|+|B'|\le  \frac{2 \varepsilon}{\alpha} n+ \frac{2 \varepsilon}{\beta} n\le \beta n/2$ times, thus we added at least $\beta n/2$ new edges. 
Therefore, we created at least $\gamma^{h-2}\beta n^{h-1}/2^{h}$ new copies of $H$. When considering all the vertices in $B'$, we may count a copy of $H$ multiple times, but clearly at most $h^2$ times, thus the number of new copies of $h$ is at least $\gamma^{h-2}\beta |B'| n^{h-1}/2^{h}h^2$. We deleted at most $\mu n$ edges, thus at most $\mu n^{h-1}$ copies of $H$. So the increase is at least $c_1\beta n^{h-1}$ for some constant $c_0$, which means $\cN(H,G^{**}) \geq \cN(H,G')+ |B| c_1\beta n^{h-1}$. Then we can finish as in Case 2.1.

\section{Proof of Theorem \ref{event}}


Our proof of Theorem \ref{event} follows the argument in \cite{mnnrw}, with a more careful analysis at some places. We will talk about injective homomorphisms from $H$ to $G$ instead of copies of $H$ in $G$. 
A homomophism from a graph $H$ to a graph $G$ is an edge preserving map from $V(H)$ to $V(G)$. 
Let $\inj(H,G)$ denote the number of injective homomorphisms from $H$ to $G$, then $\inj(H,G)=a(H)\cN(H,G)$, where $a(H)$ is the number of automorphisms of $H$. For $\beta>0$, a graph $G$ is said to be $\beta$-dense if $\deg(v)\geq (1-\beta)|V(G)|$, for every $v \in V(G)$. Recall that we denote by $h$ the number of vertices in $H$. Theorem \ref{event} is a direct consequence of the following theorem.

\begin{theorem}\label{inje}
    For every graph $H$, if $r\ge 300h^9$, then for every $\varepsilon>0$, there exists $\delta>0$ such that every $K_{r+1}$-free $n$-vertex graph $G$ with $\inj(H,G)\ge \inj(H,T(n,r))-\delta n^h$ has edit distance at most $\varepsilon n^2$ to $T(n,r)$.
\end{theorem}

The Authors of \cite{mnnrw} present an informal outline of the proof. The first idea is that most of the maps from $H$ to $G$ that are not injective homomorphisms fail due to a single edge being mapped to a non-edge. Therefore, an approximate bound on $\inj(H,G)$ can be given that depends on $e(G)$. As the Tur\'an graph contains the most edges among $K_{r+1}$-free graphs, one can obtain the result if the error terms in ``approximately'' can be controlled. The same holds in our setting. 

A key part where they use that $G$ maximizes $\inj(H,G)$ among $K_{r+1}$-free graphs is the symmetrization. They prove that the minimum degree in $G$ is large by showing that one can symmetrize a low-degree vertex to a vertex in the most copies of $H$ and increasing the number of copies of $H$. This is not a contradiction in our setting. However, we can repeatedly apply such symmetrization steps to remove all the low degree vertices, and the edit distance of the resulting graph is close to the edit distance of the original graph to $T(n,r)$.

We will use a stability theorem of F\"uredi \cite{fur}.

\begin{theorem}[F\"uredi \cite{fur}]\label{fure}
    Every $n$-vertex $K_{r+1}$-free graph $G$ contains an $r$-partite
subgraph $G_0$ such that
$e(G)-e(G_0)\le e(T(n,r))-e(G)$.
\end{theorem}

We will also use the following lemma from \cite{mnnrw}.

\begin{lemma}[\cite{mnnrw}]\label{densi}
 For every graph $H$ with at least one edge, every $0<\beta\le 1/4$, and every
$r$-partite $\beta$-dense $n$-vertex graph $G$ we have
\[\inj(H,T(n,r))-\inj(H,G)\ge 2e(H)(1-3\beta h^3)(e(T(n,r))-e(G))n^{h-2}.\]
\end{lemma}

For two graphs $H$ and $G$, and a vertex $v \in V(G)$, we denote by $\inj(v)$ the number of injective homomorphisms from $H$ to $G$ that contain $v$ in their images.
To avoid redundant repetition of arguments that have been presented in \cite{mnnrw}, we state them as facts here and refer to them in the proof. As always, in the following statements $h=|V(H)|$ and $n=|V(G)|$.
\begin{description}
    \item[Fact 1.] If $G$ is $\delta$-dense, then $\inj(H,G)\geq (1-\delta h^2)n^h$ (see Corollary 2.2 in \cite{mnnrw}). In particular, as the Tur\'an graph $T(n,r)$ is $2/r$-dense, we have $\inj(H,T(n,r))\geq (1- 2h^2/r)n^h$.
    
    \item[Fact 2.] For every vertex $v \in V(G)$, $\inj(v)\leq hn^{h-1}-(n-deg(v))n^{h-2}$.

    \item[Fact 3.] If $G'$ is obtained from $G$ by symmetrizing a vertex $v$ to another vertex $u$ in $G$, then $\inj(H,G')\geq \inj(H,G)-inj(v)+inj(u)-h^2n^{h-2}$. Also, if $G$ is $K_{r+1}$-free, then so is $G'$.

    \item[Fact 4.] Let $G'$ be a spanning $r$-partite subgraph of $G$. Then, 
    \[\inj(H,G')\ge \inj(H,G)-2e(H)(e(G)-e(G'))n^{h-2}.\]
\end{description}
Proofs of facts 2, 3, and 4, can be found in Section 3, Proof of Theorem 1.3 in \cite{mnnrw}

\begin{proof}[Proof of Theorem \ref{inje}]
Let $\beta=1/100h^6$, we are given $\varepsilon>0$, and let $\delta=\delta(\varepsilon,r,H)$ be small enough. Assume that the statement does not hold, i.e., we have an $n$-vertex graph $G$ with $\inj(H,G)\ge \inj(H,T(n,r))-\delta n^h$ and edit distance more than $\varepsilon n^2$ to $T(n,r)$. We can also assume that $n$ is sufficiently large.

First we show that we can replace $G$ by a $\beta$-dense graph $G'$. 
By Fact 2 above, for any vertex $v$ of $G$, we have
\[\inj(v)\le hn^{h-1}-(n-\deg(v))n^{h-2}.\]

By Fact 1 and our assumption on $\inj(H,G)$, we have that $\inj(H,G) \geq (1-2h^2/r-\delta)n^h$. Then, by averaging, there is a vertex $v_0$ in $G$ with 
\[\inj(v_0)\ge h(1-2h^2/r-\delta)n^{h-1}.\]


\vskip4mm

If $\deg(v)\le n-\beta n$, then $\inj(v)\le hn^{h-1}-\beta n^{h-1}$. Now we symmetrize $v$ to $v_0$ to obtain a graph $G_1$. 
By Fact 3, $G_1$ is $K_{r+1}$-free and $\inj(H,G_1)\ge \inj(H,G)-\inj(v)+\inj(v_0)-h^2n^{h-2}$. Since $r\geq 300h^9$, if $n$ is large enough and $\delta$ is small enough, then this implies $\inj(H,G_1)\ge \inj(H,G)+\beta n^{h-1}/4$.


We repeat this, as long as there is a vertex of degree at most $n-\beta n$. This can happen at most $4\delta n/\beta$ times by our assumption on $\delta$. Therefore, the resulting graph $G'$ has edit distance more than $(\varepsilon-4\delta/\beta)n^2\ge \varepsilon n^2/2$ from $T(n,r)$. In the last bound we use that $\delta$ is sufficiently small.

Now we proceed by taking an $r$-partite subgraph $G''$ of $G'$ with the most edges. Theorem \ref{fure} gives that $e(G'')\ge 2e(G')-e(T(n,r))$. Clearly $e(T(n,r))-e(G'')$ is at least half the edit distance of $G'$ and $T(n,r)$. Indeed, to obtain $G'$ from $T(n,r)$, first we delete at least $e(T(n,r))-e(G'')$ edges, then we add at most $e(T(n,r))-e(G'')$ (otherwise we obtained a $K_{r+1}$-free graph with more edges than $T(n,r)$, contradicting Tur\'an's theorem).

By Fact 4, we have
\begin{equation}\label{elev}
    \inj(H,G'')\ge \inj(H,G')-2e(H)(e(G')-e(G''))n^{h-2}.
\end{equation}

We will show that $G''$ is $2/13h^3$-dense. The proof of this statement is exactly the same as in \cite{mnnrw}, we include it for sake of completeness.
We have that $e(G'')\ge e(G')-(n^2/2-e(G'))\ge (1-1/50h^6)n^2/2$. Let $V_1,\dots,V_r$ be the parts of $G''$. As $e(G'')\le n^2-\sum_{i=1}^r |V_i|^2$, we have that $\max_{1\le i\le r}|V_i|\le n/7h^3$. By maximality of $G''$, we have that $\deg_{G''}(v)\ge \deg_{G'}(v)-\max|V_i|\ge n- \beta n -n/7h^3\ge n- 2n/13h^3$.

\vskip3mm

Now we can apply Lemma \ref{densi} to $G''$, which gives 
\[\inj(H,T(n,r))\ge \inj(H,G'')+ 2e(H)(1-6 h^3/13h^3)(e(T(n,r))-e(G''))n^{h-2}.\]


Applying the bounds $e(T(n,r))-e(G'')\ge \varepsilon n^2/4$ and (\ref{elev}), we obtain
\begin{equation*}
\begin{split}\inj(H,T(n,r)) &\ge   \inj(H,G')-2e(H)(e(G')-e(G''))n^{h-2}\\
&+ 2e(H)(e(T(n,r))-e(G''))n^{h-2}+e(H)\varepsilon n^h/52\\
&\ge \inj(H,G')+e(H)(e(G'')+e(T(n,r))-2e(G'))n^{h-2}+\varepsilon n^h/52 \\
&> \inj(H,G)+\delta n^h,\end{split}\end{equation*}

where the last inequality uses $e(G'')\ge 2e(G')-e(T(n,r))$ and that $\delta$ is sufficiently small. Clearly this contradicts our assumption on $\inj(H,G)$, completing the proof.    
\end{proof}

\section{Proof of Theorem \ref{cycles}} \label{sec4}

Given a graph $F$, denote by $\cF^\star$ the family of graphs obtained from $F$ by taking two non-adjacent vertices $u$ and $v$ of $F$ and connecting both of them to $N(u) \cup N(v)$. We say that a (possibly infinite) sequence of graphs $F_1,\dots,F_i,\dots$ is \textit{nice} if for any $i$, any $F\in \cF_i^\star$ contains some $F_j$ with $j<i$.

Observe that by definition, $F_1$ is a clique, say $K_t$. Indeed, if $F$ contains a non-adjacent pair of vertices, then there is $F\in\cF_1^\star$, and then the definition requires the existence of some $F_j$ with $j<1$ that is contained in $F$, while there are none.
Moreover, if the sequence contains graphs with chromatic number $m<t$, then the first graph with chromatic number at most $m$ must also be a clique. 
For if that graph $F_j$ is not a clique, then it has two vertices $u,v$ in the same color class. When we connect these two vertices to $N(u)\cup N(v)$, then the resulting graph $F$ has chromatic number $m$ and is contained in $\cF_j^\star$. Recall that all the graphs before $F_j$ have chromatic number larger than $m$, and hence cannot be contained in $F$.

\vskip3mm

The prime example is the sequence of odd cycles $C_{2k+1}$ (with the natural ordering). Indeed, adding any chord to an odd cycle creates a shorter odd and a shorter even cycle. Let us show some other examples. If $(F_i)$ and $(F_j')$ form nice sequences, consider the graphs $F_i+F_j'$, which are obtained by taking a copy of $F_i$ and $F_j$ and adding all the edges between them. There is a natural partial ordering: $F_i+F_j'$ is before $F_k+F_\ell'$ if they are not equal and $i\le k$, $j\le \ell$. Extending this to any total ordering, we obtain a nice sequence.

\vskip3mm

A family $\cF$ of graphs is 
 \textit{nice} if its elements can be ordered to form a nice sequence.  Let $\cC_{2k+1}^{\mathrm{odd}}=\{C_3,C_5,\dots,C_{2k+1}\}$.

\begin{proposition}\label{nice}
 Let $\cF$ be nice and $G$ be an $\cF$-free graph. If $G'$ is obtained from $G$ by symmetrizing a vertex $u$ to a non-adjacent vertex $v$, then $G'$ is $\cF$-free.  
\end{proposition}
\begin{proof}
    Consider the smallest $i$ such that $F_i$ is in $G'$. A copy of $F_i$ must contain both $u$ (since only $u$ is incident to new edges) and $v$ (since otherwise $u$ could be replaced by $v$ in the copy of $F_i$). Since $u$ and $v$ have the same neighborhood in $G'$, the graph $F$ we obtain from $F_i$ by connecting both $u$ and $v$ to $N(u) \cup N(v)$ is a subgraph of $G'$. Since $F\in \cF_i^\star$, we have that $F_j$ is a subgraph of $G'$ for some $j < i$, contradicting our choice of $i$.
\end{proof}

So far, this does not extend the known results from \cite{gyps} much. Recall that symmetrizing one of $u$ to $v$ or $v$ to $u$ does not decrease the number of copies of $H$ if $H$ is complete multipartite. This shows that if $H$ is a complete $r$-partite graph, then $\ex(n,H,\cF)=\cN(H,T)$ for some complete multipartite $n$-vertex graph $T$. However, if the smallest chromatic number of graphs in $\cF$ is $m$, then $\cF$ contains $K_m$,
thus any $\cF$-free graph is $K_m$-free, which already implies that a complete $(m-1)$-partite $n$-vertex graph is the extremal, as discussed in the introduction. Yet, surprisingly we can use Proposition \ref{nice} to prove Theorem \ref{cycles}. 

\vskip3mm

The other ingredient we use is a well-known result of Andrásfai, Erd\H os, and Sós \cite{aes}: if a $\cC_{2k+1}^{\mathrm{odd}}$-free $n$-vertex graph is not bipartite, then it contains a vertex of degree at most $2n/(2k+1)$. This assumption is missing for other nice sequences.

Recall that Theorem \ref{cycles} states that
    for every bipartite graph $H$, if $k>(2h)^{h+2}$, then
    $H$ is weakly $C_{2k+1}$-Tur\'an-stable. 
    In other words, 
    for any $\varepsilon>0$ there is $\delta>0$ such that if a $C_{2k+1}$-free $n$-vertex graph $G$ contains at least $\ex(n,H,C_{2k+1})-\delta n^h$ copies of $H$, then the edit distance of $G$ and some bipartite graph $B$ is at most $\varepsilon n^2$.


\begin{proof}[Proof of Theorem \ref{cycles}]
If $H$ has isolated vertices, they do not affect anything, thus we can assume $H$ does not have them.

Given $\varepsilon$, we let $\alpha>0$ be sufficiently small, then $\delta$ sufficiently small compared to $\alpha$ and $\varepsilon$, i.e., $\delta \ll\alpha\ll\varepsilon$. Let $n$ be sufficiently large and $G$ be an $n$-vertex $C_{2k+1}$-free graph with at least $\ex(n,H,C_{2k+1})-\delta n^h$ copies of $H$. A result of Alon and Shikhelman \cite{alon} determines when $\ex(n,H_1,F_1)=o(n^{|V(H_1)|})$, for any two graphs $H_1$ and $F_1$. It implies that
there are $o(n^{2m+1})$ copies of $C_{2m+1}$ for each $m<k$ in $G$. Therefore, by the removal lemma, there are at most $\alpha n^2/k$ edges in $G$ such that each copy of $C_{2m+1}$ contains at least one of those edges. Let $G'$ be the graph obtained by deleting those edges for each $m<k$, then clearly $G'$ is $\cC_{2k+1}^{\mathrm{odd}}$-free and $\cN(H,G')\ge \cN(H,G)-\alpha n^{|V(H)|}$.
    

    Let $\beta n^h$ be the largest number of copies of $H$ in bipartite graphs on $n$ vertices. To obtain a lower bound on $\beta$, we consider $K_{\lfloor n/2\rfloor,\lceil n/2\rceil}$. We pick the $h$ vertices one by one, each at least $\lfloor n/2\rfloor-h$ ways. Then we may count a copy multiple times, but at most $h!$ times. Therefore, $\beta\ge(\lfloor n/2\rfloor-h)^h/h!n^h> 1/(2h)^h$, using that $n$ is sufficiently large.
    
    

    By the theorem of Andrásfai, Erd\H os, and Sós we mentioned before the proof, $G'$ contains a vertex of degree at most $2n/(2k+1)$.
    Let $u,v\in G'$ be non-adjacent vertices, then by Proposition \ref{nice}, symmetrizing $u$ to $v$ results in a $\cC_{2k+1}^{\mathrm{odd}}$-free graph. 
    
    Clearly, $\cN(H,G')\ge (\beta-\alpha-\delta) n^h$, 
    thus there is a vertex $v$ of $G'$ in at least $(\beta-\alpha-\delta) hn^{h-1}$ copies of $H$.
    Let $u$ be a vertex of degree at most $2n/(2k+1)$, then $u$ is in at most $2hn^{h-1}/(2k+1)$ copies of $H$. Indeed, we pick which vertex is mapped to $u$, there are at most $2n/(2k+1)$ choices for at least one neighbor and at most $n$ choices for the other vertices. Now we symmetrize $u$ to $v$ to obtain $G''$ (in the case $u$ and $v$ are adjacent, we delete the edge between them first). We have $\cN(H,G'')\ge \cN(H,G')+(\beta-\alpha-\delta) hn^{h-1}-2hn^{h-1}/(2k+1)-n^{h-2}$, where the last term is for those copies of $H$ that contain both $u$ and $v$. By choosing $\alpha$ and $\delta$ small enough and $n$ large enough, we have $\cN(H,G'')\ge \cN(H,G')+\beta n^{h-1}/2$.


   We apply this procedure repeatedly, until we obtain a bipartite graph $B$. If there are at most $(\varepsilon -\alpha)n/2$ symmetrization steps, then the edit distance of $G'$ and $B$ is at most $(\varepsilon -\alpha)n^2$, since we change at most $2n$ edges at each step. Therefore, the edit distance of $G$ and $B$ is at most $\varepsilon n^2$. 

 If there are more than $(\varepsilon -\alpha)n/2>\varepsilon n/3$ symmetrization steps, then we gained more than $\varepsilon\beta n^h/6$ copies of $H$, thus $\cN(H,B)\ge \cN(H,G')+\varepsilon\beta n n^{h-1}/6\ge \cN(H,G)+\varepsilon\beta n^h/6-\alpha n^h$. By picking $\delta<\varepsilon \beta/6-\alpha$, we obtain a contradiction with the stability assumption.
\end{proof}

\section{Concluding remarks}

In our statements and proofs we aimed at simplicity rather than strongest possible results. In particular, we did not try to optimize the thresholds in Theorems \ref{event} or \ref{cycles}. In the case of Theorem \ref{event}, we used the threshold $300h^9$ from \cite{mnnrw}, where the Tur\'an-goodness was proved under the same assumptions. Note that there the threshold was not optimized either. These two thresholds may be the same, and it is possible that the threshold could be as low as $h+1$. In the case of Theorem \ref{cycles}, our threshold is exponential in $h$. 
We expect that the sharp threshold is actually smaller than $h$.

Another possible direction could be to study some other notions of stability. In all our bounds, $\delta$ depends linearly on $\varepsilon$. Therefore, we actually proved \textit{perfect stability}, as described in \cite{perfstb}. Another strengthening in \cite{perfstb} is that instead of Zykov symmetrization, they study a more general version. They only assume that at most $\varepsilon n^2$ edges change at each step, the value of the graph parameter does not decrease, and after some steps we arrive to a complete multipartite graph. It is possible that for some graph $H$, a symmetrization can be defined such that $\cN(H,G)$ does not decrease and the graph remains $K_{r+1}$-free, although we could not find any other example.
 Our Theorem \ref{multipa} also holds for such graphs $H$, since we use symmetrization only to obtain $G'$.

\bigskip

\textbf{Funding}: Research of Gerbner was supported by the National Research, Development and Innovation Office - NKFIH under the grants SNN 129364, FK 132060, and KKP-133819.  Research of Hama Karim was supported by the National Research, Development and Innovation Office - NKFIH under the grant FK 132060.

\end{document}